\documentclass[12pt]{article}
\usepackage{a4wide}

\usepackage[a4paper, margin=2.3cm]{geometry}
\usepackage{amsmath,amssymb,amsthm}
\usepackage{color}
\usepackage{parskip}
\usepackage{hyperref}
\usepackage{parskip}
\usepackage{setspace}
\usepackage{graphicx}

\newtheorem{theorem}{Theorem}[section]

\newtheorem{corollary}[theorem]{Corollary}
\newtheorem{lemma}[theorem]{Lemma}

\newtheorem*{definition*}{Definition}

\def\Fp{\mathbb{F}_p}

\begin{document}
\title{Some sum-product type estimates for two-variables over prime fields}

\author{Phuc D Tran \thanks{Department of Mathematics \& Sciences, American University in Bulgaria. Email: pnt170@aubg.edu}\and Nguyen Van The \thanks{VNU University of Science, Vietnam National University, Hanoi. Email: nguyenvanthe\_t61@hus.edu.vn} }

\date{}
\maketitle  
\begin{abstract}
In this paper, we use a recent method given by Rudnev, Shakan, and Shkredov (2018) to improve results on sum-product type problems due to Pham and Mojarrad (2018).
\end{abstract}

\section{Introduction and results}

Let $A$ be a set in $\mathbb{Z}$. We define the sum and product sets as follows: 
$$A+A = \{a+b\colon a, b \in A\},$$
$$A\cdot A=\{ab\colon a,b \in A\}.$$
A celebrated result of Erd\H{o}s and Szemer\'edi \cite{es} states that there is no set $A\subset\mathbb{Z}$ which has both additive and multiplicative structures. More precisely, given any finite set $A \subset \mathbb Z$, we have
$$\max\{ |A+A|, |A \cdot A|\} \gg |A|^{1+\varepsilon}$$
for some positive constant $\varepsilon$. Here, and throughout, $X \gg Y$ means that there exists $C > 0$ such that $CX \leq Y,$ and $X \gtrsim Y$ means $ X \gg \left(\log Y\right)^{-c}Y$ for some absolute constant $c>0.$

Let $\mathbb{F}_q$ be an arbitrary finite field with order $q = p^r$ for some positive integer $r$ and an odd prime $p$. Bourgain, Katz, and Tao \cite{bourgain-katz-tao} showed that given any set $A \subset \mathbb F_{p}$ with $p$ prime and $p^{\delta} < |A| < p^{1-\delta}$ for some $\delta >0,$ one has
$$\max\{ |A+A|, |A \cdot A|\} \geq C_{\delta}|A|^{1+\varepsilon},$$ for some $\varepsilon=\varepsilon(\delta) > 0$. 
Note that the relation between $\varepsilon$ and $\delta$ is difficult to determine.

Suppose that $|A + A| = m$ and $|A \cdot A | = n$, using Fourier analytic methods, Hart, Iosevich, and Solymosi \cite{his} gave an explicit bound over arbitrary finite fields as follows
\begin{equation}\label{eq:his}
|A|^3 \ll \frac{  m^2 n |A| }{ q} +  q^{1/2} mn.
\end{equation}

Note that this bound is only non-trivial when $|A|\gg q^{1/2}$. Using a graph theoretic method, Vinh \cite{vinh} obtained an improvement and as a consequence, derived a stronger bound for large sets, namely,
\begin{equation}\label{eq-x-x}
|A|^2 \leq \frac{mn|A|}{q} + q^{1/2} \sqrt{ m n } .
\end{equation}

The following are direct consequences of (\ref{eq-x-x}).

\begin{enumerate} 
\item If $q^{1/2} \ll |A| < q^{2/3}$, then 
\begin{equation}\label{hh1}
\max \{ |A + A | , |A \cdot A | \} \gg \frac{ |A|^2 }{q^{1/2}} .
\end{equation}
\item If $q^{2/3} \leq |A| \ll q$, then 
\begin{equation}\label{hh2}
\max \{ |A + A | , |A \cdot A | \} \gg ( q |A| )^{1/2} .
\end{equation}
\end{enumerate}
Notice that these bounds were first proved by Garaev \cite{Ga} over prime fields by using exponential sums.

Let $G$ be a subgroup of $\mathbb{F}^{*}$, and an arbitrary function $g:G \rightarrow \mathbb{F}^{*}$. Define $$\mu(g):=\max_{t \in \mathbb{F}^{*}} |\{x \in G: g(x) = t\}|.$$

Hegyv\'{a}ri and Hennecart \cite{heg} studied generalizations of (\ref{hh1}) and (\ref{hh2}) for certain families of polynomials by using methods from spectral graph theory. The precise statements of their results can be stated in two following theorems.

\begin{theorem}[\textbf{Hegyv\'{a}ri and Hennecart}, \cite{heg}]\label{thm1}
Let $G$ be a subgroup of $\mathbb{F}_p^*$. Consider the function $f(x,y)=g(x)(h(x)+y)$ on $G\times \mathbb{F}_p^*$, where $g,h\colon G\to \mathbb{F}_p^*$ are arbitrary functions. Define $m=\mu(g\cdot h)$. For any subsets $A\subset G$ and $B,C\subset \mathbb{F}_p^*$, we have
\[\left\vert f(A,B)\right\vert \left\vert B\cdot C\right\vert\gg \min\left\lbrace\frac{|A||B|^2|C|}{pm^2}, \frac{p|B|}{m}\right\rbrace.\] 
\end{theorem}
\bigskip
\begin{theorem}[\textbf{Hegyv\'{a}ri and Hennecart}, \cite{heg}]\label{2.2}
Let $G$ be a subgroup of $\mathbb{F}_p^*$. Consider the function $f(x,y)=g(x)(h(x)+y)$ on $G\times \mathbb{F}_p^*$, where $g,h\colon G\to \mathbb{F}_p^*$ are arbitrary functions. Define $m=\mu(g)$. For any subsets $A\subset G$, $B,C \subset \mathbb{F}_p^*$, we have
\[|f(A,B)||B+C|\gg \min\left\lbrace  \frac{|A||B|^2|C|}{pm^2}, \frac{p|B|}{m}\right\rbrace.\]
\end{theorem}

In \cite{PhMo}, Pham and Mojarrad used a point-plane incidence bound due to Rudnev \cite{Ru} to improve Theorems \ref{thm1} and \ref{2.2} when sets $A$, $B$, and $C$ are not too big. 

\begin{theorem}[\textbf{Pham-Mojarrad}, \cite{PhMo}]\label{thm1*}
Let $f(x,y)=g(x)(h(x)+y)$ be a function defined on $\mathbb{F}_p^*\times \mathbb{F}_p^*$, where $g,h\colon \mathbb{F}_p^*\to \mathbb{F}_p^*$ are arbitrary functions. Define $m=\mu(g\cdot h)$. For any subsets $A, B,C\subset \mathbb{F}_p^*$ with $|A|, |B|, |C|\le p^{5/8}$, we have
\[\max\left\lbrace |f(A, B)|, |B\cdot C|\right\rbrace\gg \min\left\lbrace\frac{|A|^{\frac{1}{5}}|B|^{\frac{4}{5}}|C|^{\frac{1}{5}}}{m^{\frac{4}{5}}}, \frac{|B||C|^{\frac{1}{2}}}{m}, \frac{|B||A|^{\frac{1}{2}}}{m}, \frac{|B|^{\frac{2}{3}}|C|^{\frac{1}{3}}|A|^{\frac{1}{3}}}{m^{\frac{2}{3}}}\right\rbrace.\]
\end{theorem}

\begin{theorem}[\textbf{Pham-Mojarrad}, \cite{PhMo}]\label{thm2*}
Let $f(x,y)=g(x)(h(x)+y)$ be a function defined on $\mathbb{F}_p^*\times \mathbb{F}_p^*$, where $g, h\colon \mathbb{F}_p^*\to \mathbb{F}_p^*$ are arbitrary functions. Define $m=\mu(g)$. For any subsets $A, B,C\subset \mathbb{F}_p^*$ with $|A|, |B|, |C|\le p^{5/8}$, we have
\[\max\left\lbrace |f(A, B)|, |B+ C|\right\rbrace\gg \min\left\lbrace\frac{|A|^{\frac{1}{5}}|B|^{\frac{4}{5}}|C|^{\frac{1}{5}}}{m^{\frac{4}{5}}}, \frac{|B||C|^{\frac{1}{2}}}{m}, \frac{|B||A|^{\frac{1}{2}}}{m}, \frac{|B|^{\frac{2}{3}}|C|^{\frac{1}{3}}|A|^{\frac{1}{3}}}{m^{\frac{2}{3}}}\right\rbrace.\]
\end{theorem}

It follows from Theorems \ref{thm1*} and \ref{thm2*} that for $A\subset \mathbb{F}_p$ if $|A|\le p^{5/8}$, then we have 
\[\max\{|A+A|, |A\cdot A|\}\gg |A|^{6/5}.\]
It is worth noting that the best current bound in the literature is due to Rudnev, Shakan, and Shkredov \cite{RuSh1}. More precisely, they proved that for $A\subset \mathbb{F}_p$ with $|A|\le p^{18/35}$, we have 
\[\max\{|A+A|, |A\cdot A|\}\gg |A|^{11/9-o(1)}.\]
The main purpose of this paper is to employ the method of Rudnev, Shakan, and Shkredov in \cite{RuSh1} to improve further Theorems \ref{thm1*} and \ref{thm2*}. Our first main result is as follows. 

\begin{theorem} \label{maintheorem2}
Let $f_1(x,y)=g_1(x)(h_1(x)+y), f_2(x,y)=g_2(x)(h_2(x)+y)$ be the functions defined on $\mathbb{F}_p^{*} \times \mathbb{F}_p^{*}$, where $g_1,h_1,g_2,h_2: \mathbb{F}_p^{*} \rightarrow \mathbb{F}_p^{*}$ are arbitrary functions. Define $m= \max \{\mu(g_1),\mu(g_2)\}$. For any subsets $A,B,C,D \subset \mathbb{F}_p^{*}$, with $|A| \leq |B|\leq p^{3/5}$ and $|D| \leq |C| \leq p^{3/5}$, we have
$$\max \left\{|f_1(A,B)|,|f_2(D,C)|,|B-C|\right\} \gtrsim
\frac{|B|^{\frac{23}{36}}|C|^{\frac{13}{36}}|A|^{\frac{7}{36}}|D|^{\frac{1}{36}}}{m^{\frac{8}{9}}}.$$
\end{theorem}
\begin{theorem}\label{maintheorem3}  Let $f_1(x,y)=g_1(x)(h_1(x)+y), f_2(x,y)=g_2(x)(h_2(x)+y)$ be the functions defined on $\mathbb{F}_p^{*} \times \mathbb{F}_p^{*}$, where $g_1,h_1,g_2,h_2: \mathbb{F}_p^{*} \rightarrow \mathbb{F}_p^{*}$ are arbitrary functions. Define $m=\mu(g_1)=\mu(g_2)$. For any subsets $A,B,C,D \subset \mathbb{F}_p^{*}$, with $|A|\leq |B|\leq p^{3/5},|D|\leq |C| \leq p^{3/5}$, we have
$$\max\left\{|f_1(A,B)|,|f_2(D,C)|,|B+C|\right\} \gtrsim \frac{|C|^{\frac{5}{18}}|B|^{\frac{13}{18}}|A|^{\frac{1}{6}}|D|^{\frac{1}{18}}}{m^{\frac{8}{9}}}.$$
\end{theorem}

The following are consequences of theorems \ref{maintheorem2} and \ref{maintheorem3}. Firstly, when $A=B=C=D$ and $f_1 \equiv f_2$, we have: 
\begin{corollary}\label{cor1} Let $f(x,y)=g(x)(h(x)+y)$ be a function defined on $\mathbb{F}_p^{*} \times \mathbb{F}_p^{*}$, where $g,h: \mathbb{F}_p^{*} \rightarrow \mathbb{F}_p^{*}$ are arbitrary functions. Define $m=\mu(g)$. For any subset $A \subset \mathbb{F}_p^{*}$, with $|A| \leq p^{3/5}$, we have
$$\max\{|f(A,A)|,|A \pm A|\} \gg |A|^{\frac{11}{9}-o(1)}.$$
\end{corollary}
Moreover, when $ A=D, f_1 \equiv f_2$, we improved a result of Pham and Mojarrad in $\cite{PhMo}$.
\begin{corollary} \label{cor2}
Let $f(x,y)=g(x)(h(x)+y)$ be a function defined on $\mathbb{F}_p^{*} \times \mathbb{F}_p^{*}$, where $g,h: \mathbb{F}_p^{*} \rightarrow \mathbb{F}_p^{*}$ are arbitrary functions. Define $m=\mu(g)$. For any subsets $A,B,C \subset \mathbb{F}_p^{*}$, with $|A|\leq |B|,|C| \leq p^{3/5}$, we have
$$\frac{|B|^{13/18}|C|^{5/18}|A|^{2/9}}{m^{\frac{8}{9}}} \lesssim \max \{|f(A,B)|,|B + C|\},$$
and 
$$\frac{|B|^{23/36}|C|^{13/36}|A|^{2/9}}{m^{\frac{8}{9}}} \lesssim \max \{|f(A,B)|,|B - C|\}.$$
 \end{corollary}
In the following theorem, we provide the multiplicative version of Theorem \ref{maintheorem2}.
\begin{theorem}\label{maintheorem4} Let $f_1(x,y)=g_1(x)(h_1(x)+y), f_2(x,y)=g_2(x)(h_2(x)+y)$ be the functions defined on $\mathbb{F}_p^{*} \times \mathbb{F}_p^{*}$, where $g_1,h_1,g_2,h_2: \mathbb{F}_p^{*} \rightarrow \mathbb{F}_p^{*}$ are arbitrary functions. Define $m= \max \{\mu(g_1 \cdot h_1),\mu(g_2 \cdot h_2)  \}$. For any subsets $A,B,C,D \subset \mathbb{F}_p^{*}$, with $|A|\leq |B|\leq p^{3/5},|D|\leq |C| \leq p^{3/5}$, we have
$$\max\left\{|f_1(A,B)|,|f_2(D,C)|,|B \cdot C|\right\} \gtrsim \frac{|C|^{\frac{5}{18}}|B|^{\frac{13}{18}}|A|^{\frac{1}{6}}|D|^{\frac{1}{18}}}{m^{\frac{8}{9}}}.$$
\end{theorem}

\begin{corollary} \label{cor4} Let $f(x,y)=g(x)(h(x)+y)$ be a function defined on $\mathbb{F}_p^{*} \times \mathbb{F}_p^{*}$, where $g,h: \mathbb{F}_p^{*} \rightarrow \mathbb{F}_p^{*}$ are arbitrary functions. Given $m=\mu(g \cdot h)$ is finite. For any subset $A \subset \mathbb{F}_p^{*}$, with $|A| \leq p^{3/5}$, we have:
$$\max\left\{|f(A,A)|,|A.A|\right\} \gg |A|^{\frac{11}{9}-o(1)}.$$
\end{corollary}

\begin{corollary} \label{cor5}
Let $f(x,y)=g(x)(h(x)+y)$ be a function defined on $\mathbb{F}_p^{*} \times \mathbb{F}_p^{*}$, where $g,h: \mathbb{F}_p^{*} \rightarrow \mathbb{F}_p^{*}$ are arbitrary functions. Define $m=\mu(g.h)$. For any subsets $A,B,C \subset \mathbb{F}_p^{*}$, with $|A| \leq |B|,|C| \leq p^{3/5}$, we have
$$\frac{|C|^{5/18}|B|^{13/18}|A|^{2/9}}{m^{8/9}} \lesssim \max \left\{|f(A,B)|,|B \cdot C|\right\}.$$
\end{corollary}
Given $f_1(x,y)=x(1+y)$ and  $f_2(x,y)=x(1-y)$, we obtain a concise version of \cite[Theorem $2$]{Wa} by Warren with a stronger condition.  
\begin{corollary} \label{cor6}
For any subsets $A,B,C,D \subset \mathbb{F}_p^{*}$, with $|A|\leq |B|\leq p^{3/5},|D|\leq |C| \leq p^{3/5}$, we have:
$$\max\left\{|A(1+B)|,|D(1-C)|,|B \cdot C|\right\} \gtrsim |C|^{\frac{5}{18}}|B|^{\frac{13}{18}}|A|^{\frac{1}{6}}|D|^{\frac{1}{18}}.$$
\end{corollary}

Finally, combining Theorems \ref{maintheorem2} and \ref{maintheorem4}, we are able to improve \cite[Theorem $1.10$]{PhMo}.
\begin{theorem} \label{maintheorem5}
Let $f(x,y)=g(x)(h(x)+y)$ be a function defined on $\mathbb{F}_p^{*} \times \mathbb{F}_p^{*}$, where $g,h: \mathbb{F}_p^{*} \rightarrow \mathbb{F}_p^{*}$ are arbitrary functions. Given $\mu(g)$ is finite. For any subsets $A \subset \mathbb{F}_p^{*}$, with $|A| \leq p^{3/5}$, satisfying: 
$$\min \left\{|A+A|,|A \cdot A|\right\} \leq |A|^{\frac{6}{5}-\epsilon},$$
for some $\epsilon>0$. Then, we have
$$|f(A,A)| \gg |A|^{\frac{8}{5}-\frac{3}{25}+\frac{2\epsilon}{5}}.$$
\end{theorem}

\section{Preliminaries}
\noindent Let $G$ be an abelian group, and let $A,B$ be two finite subsets of $G.$ For any real number $n>1$, we define the representation functions:
$$r_{A-B}(x):= \# \{(a,b) \in A\times B: x = a - b,x\in G \},$$
$$E_n(A,B):= \sum_{x} r^n_{A-B}(x),\, E_n(A):= E_n(A,A).$$
Similarly, 
 $$r_{A/B}(x):= \# \{(a,b) \in A\times B: x = a . b^{-1}, x\in G \},$$
 $$E_n^{\times}(A,B):= \sum_{x} r^n_{A/B}(x),\, E_n^{\times}(A):= E_n^{\times}(A,A).$$
The initial idea of this paper is to use the Rudnev point-plane incidence's theorem and the theory of higher order energies to optimize the bounds on $E_4(C,D), E_4^{\times}(C,D)$ for some small sets $C,D$. 
\begin{theorem}\label{theo4}\textsl{(Rudnev, \cite{Ru})}  Let $\mathcal{R}$ be a set of points in $\mathbb{F}_p^3$ and let $\mathcal{S}$ be a set of planes
in $\mathbb{F}_p^3$, with $|\mathcal{R}| \ll |\mathcal{S}|$ and $|\mathcal{R}| \ll p^2.$ Assume that there is no line containing $k$ points of $\mathcal{R}.$ Then
\[ I(\mathcal{R},\mathcal{S}) \ll |\mathcal{R}|^{1/2}|\mathcal{S}|+k\,|\mathcal{S}|.\]
\end{theorem}

\begin{lemma}\label{lemmasum}
Let $f(x,y)=g(x)(h(x)+y)$ be a function defined on $\mathbb{F}_p^{*} \times \mathbb{F}_p^{*}$, where $g,h: \mathbb{F}_p^{*} \rightarrow \mathbb{F}_p^{*}$ are arbitrary functions. Define $m=\mu(g)$. For any subsets $A,B,C \subset \mathbb{F}_p^{*}$, with $|A|\ll |B|$ and $|A|,|B|,|C| \leq p^{3/5}$, we have:
$$E_4(B,C):=\sum_{x}r_{B-C}^4(x) \ll  m^4\cdot\min\left\{\frac{|f(A,B)|^3|C|^2}{|A|}, \frac{|f(A,B)|^2|C|^3}{|A|}\right\}\cdot\log|A|.$$
\end{lemma}

\begin{proof} For $1 \leq k \leq \text{min}\{|A|,|B|,|C|\}$, let $n_k:=\left|X_k:=\{ x \in B-C:r_{B-C}(x) \geq k\}\right|.$ By a dyadic decomposition, there exist a number $k$ such that $E_4(B,C) \ll k^4n_k.$ We consider the following equations:
\begin{align} \label{eq1}
g(a)(x+c+h(a))-f(a,b)=0
\Leftrightarrow c=\frac{f(a,b)}{g(a)}-x-h(a),
\end{align}
where $a \in A, b \in B, x \in X_k, c \in C$. 

Let $M$ be the number of solutions to the equation \eqref{eq1}, we have $M \ge k|A|n_k$ by the definition of $X_k.$ On the other hand, by using Cauchy-Schwartz inequality on each equation in (\ref{eq1}), we obtain that: 
\begin{align} \label{E_1}
M &\leq |f(A,B)|^{1/2}.\sqrt{\left|\{(a,x,c,a',x',c')\in (A \times X_k \times C)^2: g(a)(x+c+h(a))=g(a')(x'+c'+h(a')\}\right|} \notag\\
&=:|f(A,B)|^{1/2}\sqrt{\mathcal{E}_1},
\end{align}
and also 
\begin{align} \label{E_2}
M &\leq |C|^{1/2}\sqrt{\left|\{(a,x,f,a',x',f') \in (A \times X \times f(A,B))^2:\frac{f(a,b)}{g(a)}-x-h(a)=\frac{f(a',b')}{g(a')}-x-h(a')\}\right| } \notag  \\
&=:|C|^{1/2}\sqrt{\mathcal{E}_2}.
\end{align} 
Firstly, we obtain the upper bound on $E_4(B,C)$ via $\mathcal{E}_1$. We consider the following cases. 

\textbf{Case 1:} If $|X_k||A||C| \gg p^2,$ our assumptions and the definition of $X_k$ provide that $k \ll \min\{ |A|,|B|,|C|\}$, $|A| \ll |B| \ll |f(A,B)|$, and $k|X_k| \ll |B||C|$.

If $k^3 \ll \dfrac{|f(A,B)|^2|C|^2}{|A||B|}$, then 
\[ |X_k|k^4 \ll |B||C|k^3 \ll \dfrac{|f(A,B)|^2|C|^3}{|A|}\] and the  
result follows from $E_4(B,C) \ll k^4|X_k|.$ We therefore assume that $k^3 \gg \dfrac{|f(A,B)|^2 |C|^2}{|A||B|}$, then 
\[|B||C||A||C| \gg |X_k|t|A||C| \gg kp^2 \gg p^2\dfrac{|f(A,B)|^{2/3}. |C|^{2/3}}{|A|^{1/3}.|B|^{1/3}}.\]
This leads to $\left(|A||B||C|\right)^{4/3} \gg p^2 \left|f(A,B)\right|^{2/3} \gg  p^2|B|^{2/3}.$ It follows that
 \[ \max\left(|A|,|B|,|C|\right) \gg p^{3/5}.\]
This contradicts our setting-up condition.

\textbf{Case 2:} If $|X_k||A||C| \ll p^2,$ we define the set of points $\mathcal{R}_1$ and the set of planes $\mathcal{S}_1$ as following. 
\begin{align*}
\mathcal{R}_1&:=\left\{ (x,g(a'),g(a').(c'+h(a')): (a',c',x) \in A \times C \times X_k\right\}, \\
\mathcal{S}_1&:=\left\{g(a)\cdot X-x' \cdot Y-Z+g(a).(c+h(a))=0: (a,c,x') \in A \times C \times X_k\right\}.
\end{align*}

Note that $\mu(g)=m$ or there are at most $m$ different values of $a$ satisfying the equation $g(a)=t$ for all $t \in \mathbb{F}^*$. Therefore, we have
\[ \mathcal{E}_1 \leq m^2 \cdot I(\mathcal{R}_1,\mathcal{S}_1),\]
in which, $I(\mathcal{R}_1,\mathcal{S}_1)$ is the number incidences between $\mathcal{R}_1$ and $\mathcal{S}_1$.

To apply Theorem \ref{theo4}, we need to find an upper bound on the maximum number of collinear points in $\mathcal{R}_1$. The projection of $\mathcal{R}_1$ into the last two coordinates is the set $\mathcal{T}= \{(g(a'),g(a')(c'+h(a')): a' \in A, c \in C\}$. The set $\mathcal{T}$ can be covered by at most $|A|$ lines of the form $X =g(a')$ with $a' \in A$, where each line contains $|C|$ points of $\mathcal{T}$. Therefore, a line in $ \mathbb{F}_p^3$ contains at most $\max \{|A|, |C|\}$ points of $\mathcal{R}_1$, unless it is vertical, in which case it contains at most $|X_k|$ points. All implies the maximum number of collinear poins in $\mathcal{R}_1$ is at most $M_1 := \max\left\{ |A|, |C|, |X_k| \right\}.$ We fall into the different situations as follows.

\textit{Case 2.1.} If $M_1=\max\left(|X_k,|A|,|C|\right)\gg (|A||X_k||C|)^{1/2}$, then since $k \ll \min\left(|A|,|B|,|C|\right)$ and $|A|\ll |B| \ll |f(A,B)|$, we obtain the followings:\\
\begin{itemize}
\item[1.] If $M_1=|C|$ and $|X_k||A|\ll |C|$, then
\[ E_4(B,C)\ll |X_k|k^4 \ll |C|k^3 \ll \dfrac{|f(A,B)|^2}{|A|}|C|^3.\]

\item[2.] If $M_1=|X_k|$ and $|A||C|\ll |X_k|$, then $|A||C|k \ll|X_k|k \ll |B||C|$, that implies $|A|k \ll |B|.$ It follows that
\[ E_4(B,C)\ll |X_k|k^4 \ll \frac{|X_k|k}{|A|}(|A|k)k^2 \ll \frac{|B||C|}{|A|}|B|k^2 \ll \frac{|f(A,B)|^2}{|A|}|C|^3.\]

\item[3.] If $M_1=|A|$ and $|X_k||C| \ll |A|,$ then
$$E_4(B,C) \ll |X_k|k^4 \ll |A|k^3 \ll \frac{|f(A,B)|^2}{|A|}|C|^3.$$
\end{itemize}

\textit{Case 2.2.} If $M_1=\max\left(|X_k,|A|,|C|\right)\ll (|A||X_k||C|)^{1/2}$ and $|\mathcal{R}_1| =|A||C||X_k| \ll p^2$, applying Theorem \ref{theo4}, we obtain 
\begin{align*}
I(\mathcal{R}_1,\mathcal{S}_1) &\ll (|A|.|X_k|.|C|)^{3/2}+ \max\left\{|A|,|C|,|X_k|\right\}(|C||A||X_k|) \\
&\ll (|A||X_k||C|)^{3/2}.
\end{align*}

The sub-cases 2.1 and 2.2 together imply that for $|X_k|.|A|.|C| \ll p^2,$ then either
$$I(\mathcal{R}_1,\mathcal{S}_1) \ll (|A|.|X_k|.|C|)^{3/2}$$
or 
$$E_4(B,C) \ll \frac{|f(A,B|^2}{|A|}.|C|^3.$$
Therefore, collecting all above cases and \eqref{E_1}, we get either
$$E_4(B,C)\ll \frac{|f(A,B)|^2.|C|^3}{|A|}$$
or  $k|A|n_k \leq m|f(A,B)|^{1/2}I(\mathcal{R}_1,\mathcal{S}_1)\ll m|f(A,B)|^{1/2}(|A||C|n_k)^{3/2},$ i.e
\[n_k \ll \frac{m^4}{k^4}\frac{|f(A,B)|^2|C|^3}{|A|}.\] 

Now, we use $\mathcal{E}_2$ to obtain another bound on $M$. Similarly, we define a set of points $\mathcal{R}_2$ and a set of planes $\mathcal{S}_2$ as follows:
\begin{align*}
\mathcal{R}_2 &= \left\{ (f,\frac{1}{g(a')}, h(a')+x'): (a',x',f) \in A \times X_k \times f(A,B)\right\},\\ 
\mathcal{S}_2 &= \left\{ \frac{1}{g(a)}\cdot X-f'\cdot Y+Z-h(a)-x=0: (a,x,f') \in A \times X_k \times f(A,B)\right\}.
\end{align*}

Clearly, $|\mathcal{S}_2|=|\mathcal{R}_2| \ll |f(A,B)||A||X_k|$ and the maximal number of collinear points in $\mathcal{R}_2$ is at most  $\max \{|f(A,B)|,|A|,|X_k|\}$. \\
As same as the procedure above, since $|A|,|B|,|C| \leq p^{3/5}$, applying Theorem \ref{theo4}, we obtain either
\[ I(\mathcal{R}_2,\mathcal{S}_2) \ll (|f(A,B)||A||X_k|)^{3/2}\]
or 
\[ E_4(B,C) \leq \frac{|f(A,B)|^3.|C|^2}{|A|}.\]
Together with \eqref{E_2}, we get 
\[ k^4|A|^4n_k^4 \ll M^4 \ll m^4|C|^2|f(A,B)|^3|A|^3n_k^3\]
This implies 
\[ n_k \ll \frac{1}{k^4}\frac{m^4|f(A,B)|^3|C|^2}{|A|}.\]
All above implies either
\[ E_4(B,C) \ll \min\left\{ \frac{|f(A,B)|^3|C|^2}{|A|}, \frac{|f(A,B)|^2|C|^3}{|A|}\right\}\]
or
\[n_k \ll \frac{m^4}{k^4} \cdot\min\left\{ \frac{|f(A,B)|^3|C|^2}{|A|}, \frac{|f(A,B)|^2|C|^3}{|A|} \right\}.\]
Following after dyadic summation in $k$, we finally get
\[ E_4(B,C):=\sum_{x}r_{B-C}^4(x) \ll \left( m^4 \cdot \min\left\{ \frac{|f(A,B)|^3|C|^2}{|A|}, \frac{|f(A,B)|^2|C|^3}{|A|}\right\}\right)\cdot\text{log}|A|,\]
which completes the proof of Lemma \ref{lemmasum}.
\end{proof}

\textbf{Remark.} If $B=C$ and taking the later side of the minimum, the above inequality reduces to
\begin{align*}
E_4(B) &\lesssim m^4 \cdot \frac{|f_1(A,B)|^2|B|^3}{|A|}, \\
E_4(C) &\lesssim m^4 \cdot \frac{|f_2(D,C)|^2|C|^3}{|D|}
\end{align*}
for any $|A|,|B|,|C|,|D| \leq p^{3/5}$ and the functions $f_1(x,y)=g_1(x)(h_1(x)+y),f_2(x,y)=g_2(x)(h_2(x)+y)$ with $\mu(g_1),\mu(g_2) \leq m$.

Similarly, we obtain the bound on the product-representation function
.
\begin{lemma} \label{lemmapro}
Let $f(x,y)=g(x)(h(x)+y)$ be a function defined on $\mathbb{F}_p^{*} \times \mathbb{F}_p^{*}$, where $g,h: \mathbb{F}_p^{*} \rightarrow \mathbb{F}_p^{*}$ are arbitrary functions. Define $m=\mu(g.h)$. For any subsets $A,B,C \subset \mathbb{F}_p^{*}$, with $|A|,|B|,|C| \leq p^{3/5}$, we have:
$$E_4^{\times}(B,C):=\sum_{x}r_{B/C}^4(x) \ll \left( m^4 \cdot \min \left\{\frac{|f(A,B)|^3|C|^2}{|A|}, \frac{|f(A,B)|^2|C|^3}{|A|}\right\}\right)\cdot\log{|A|}.$$
\end{lemma}

\begin{proof} For $1 \leq k \leq \text{min}\{|A|,|B|,|C|\}$, let $s_k:=|Y_k:=\{ x \in B/C:r_{B/C}(x) \geq k\}|.$ By a dyadic decomposition, there exist a number $k$ such that $E^{\times}_4(B,C) \ll |Y_k|k^4.$ We consider the following equations
\begin{align} \label{proeq1}
g(a)(x \cdot c+h(a))-f(a,b)=0 
\Leftrightarrow c = x^{-1}\left(\frac{f(a,b)}{g(a)}-h(a)\right),
\end{align}
where $a \in A, b \in B, x \in Y_k, c \in C$. Clearly, there are $M' \geq k|A|.n_k$ solutions to the above equations. By using Cauchy-Schwarz inequality on each of them, we obtain that 
\begin{align}\label{E_3}
M' &\leq |f(A,B)|^{1/2}\sqrt{\left|\left\{(a,x,c,a',x',c')\in (A \times Y_k \times C)^2: g(a)(x\cdot c+h(a))=g(a')(x' \cdot c'+h(a')\right\}\right|} \notag\\
&=:|f(A,B)|^{1/2}\sqrt{\mathcal{E}_3}, 
\end{align}
and 
\begin{align}\label{E_4}
M' &\leq |C|^{1/2}\sqrt{\left|\left\{(a,x,f,a',x',f') \in (A \times Y_k \times f(A,B))^2:\frac{1}{x}\left(\frac{f(a,b)}{g(a)}-h(a)\right)=\frac{1}{x'}\left(\frac{f(a',b')}{g(a')}-h(a')\right)\right\}\right|} \notag \\
&=: |C|^{1/2}\sqrt{\mathcal{E}_4}
\end{align} 

Firstly, we obtain the upper bound on $\mathcal{E}_3$. Define the set of points $\mathcal{R}_1$ and the set of planes $\mathcal{S}_1$ as following:
\begin{align*}
\mathcal{R}_1 &= \left\{ \left(x,g(a')c',g(a')h(a')\right): (a',c',x') \in A \times C \times Y_k\right\},\\
\mathcal{S}_1 &= \left\{g(a)c \cdot X - x' \cdot Y - Z + g(a)h(a)=0: (a,c,x') \in A \times C \times Y_k\right\}.
\end{align*}

Note that $m=\mu(g \cdot h)$ or there are at most $m$ different values of $a$ satisfying the equation $g(a)h(a)=t$  for any $t \in \mathbb{F}^{*}$. Therefore, we have

\[ E_3 \leq m^2 \cdot I(\mathcal{R}_1,\mathcal{S}_1)\]

in which, $I(\mathcal{R}_1,\mathcal{S}_1)$ is the number of incidences between $\mathcal{R}_1$, and  $\mathcal{S}_1$.

Similar to the Lemma \ref{lemmasum} and applying the Theorem \ref{theo4}, we also get either 
$$s_k\ll \frac{m^4}{k^4}\dfrac{|f(A,B)|^2|C|^3}{|A|},$$
or 
$$E^{\times}_4(B,C) \ll \frac{|f(A,B)|^2|C|^3}{|A|}.$$

We now use $\mathcal{E}_4$ to obtain another bound on $M'$. By the same procedure, we define a set of points $\mathcal{R}_2$ and a set of planes $\mathcal{S}_2$ as
\begin{align*} 
\mathcal{R}_2 &= \left\{ \left(f,\frac{1}{g(a').x'}, \frac{h(a')}{x'}\right): (a',x',f) \in A \times Y_k \times f(A,B)\right\}, \\
\mathcal{S}_2 &= \left\{ \frac{1}{xg(a)}\cdot X - f'\cdot Y+Z- \frac{h(a)}{x} = 0 : (a,x,f') \in A \times Y_k \times f(A,B)\right\}.
\end{align*}
Note that if $u = 1/\left(xg(a)\right) , v = x^{-1}h(a)$, then $g(a)h(a)=vu^{-1}$. Therefore, once again, we have
$$E_4 \leq m^2 \cdot I(\mathcal{R}_2,\mathcal{S}_2)$$
in which, $I(\mathcal{R}_2,\mathcal{S}_2)$ is the number of incidences between $\mathcal{R}_2$ and $\mathcal{S}_2$. Moreover, $|\mathcal{S}_2|=|\mathcal{R}_2| \ll |f(A,B)||A||Y_k|$. Again, since $|A|,|B|,|C| \leq p^{3/5}$, we have $|f(A,B)||A||Y_k| \ll p^2$, and then either
$$E^{\times}_4 (B,C) \ll \frac{|f(A,B)|^3|C|^2}{|A|},$$
or
$$s_k \ll \frac{m^4}{k^4}\frac{|f(A,B)|^3|C|^2}{|A|}.$$
All above implies either
$$E^{\times}_4(B,C) \ll \text{min} \left\{ \frac{|f(A,B)|^3|C|^2}{|A|}, \frac{|f(A,B)|^2|C|^3}{|A|} \right\},$$
or
$$s_k \ll \frac{m^4}{k^4}.\text{min} \left\{ \frac{|f(A,B)|^3|C|^2}{|A|}, \frac{|f(A,B)|^2|C|^3}{|A|} \right\}.$$
Following after dyadic summation in $k$, we finally get 
$$E^{\times}_4(B,C):=\sum_{x}r_{B/C}^4(x) \ll \left( m^4 \cdot \min\left\{ \frac{|f(A,B)|^3|C|^2}{|A|}, \frac{|f(A,B)|^2|C|^3}{|A|}\right\}\right) \cdot \log|A|,$$
which completes the proof of Lemma \ref{lemmapro}.
\end{proof}

\section{Proof of Theorem \ref{maintheorem2}}

Let $P \subset B-C$ be a set of popular differences, defined as follows 
\[ P := \left\{ x \in B - C : r_{B-C}(x) \geq \frac{|B||C|}{2|B-C|} \right\}.\]  
We further obtain 
\begin{align} \label{conP1}
\left| \{ (b_1,c_1) \in B \times C: b_1-c_1 \in P\} \right| \gg |B||C|.
\end{align}
Consider the equation
\begin{align} \label{eq2}
b-c=(a-c)-(a-b)=(d-c)-(d-b).
\end{align} 
Suppose $x=a-c$ and $y=d-c$ are in $P$, while $u=a-b, v=d-b$ are both in $B-B$ (*). By the condition \eqref{conP1}, the equation \eqref{eq2} has $N \geq \frac{|P|^2}{|C|}.|B|\gg |B|^3.|C|$ solutions $(a,b,c,d)$.\\
We define an equivalent relation on $B \times B \times C \times B$ as follows:
$$(a,b,c,d) \sim (a',b',c',d') \Leftrightarrow (a,b,c,d)=(a'+t,b'+t,c'+t,d'+t),$$ 
for some $t \in (B-B) \cap (C-C)$, and we also denote the equivalent class of $(a,b,c,d)$ by $[a,b,c,d]$. 

Clearly, if $(a,b,c,d) \sim (a',b',c',d')$ and $(a,b,c,d)$ is a solution of equation $(\ref{eq2})$, then $(a',b',c',d')$ is also a solution. Thus, we can decompose $N$ into the sum over each equivalent class, which satisfies (*).\\ 
$$N = \sum_{[a,b,c,d]} r([a,b,c,d]),$$
in which $r([a,b,c,d])$ is number of elements in the $[a,b,c,d]$-equivalent class. 
Applying the Cauchy-Schwartz inequality, we obtain
\[ N^2 \leq X^{*}\cdot\left(\sum_{[a,b,c,d]} r^2([a,b,c,d])\right),\]
in which, $X^{*}$ is number of equivalent classes, which satisfying (*).  \\
Moreover, each equivalent class is defined uniquely by any three of five $x,y,u,v,w$, and each equivalent class provides a distinct solution of system
$$x,y \in P, u,v \in B-B, w \in B-C: x-u=y-v=w.$$
It implies that
$$X^{*} \leq \mathcal{X} := \left|\left\{ x,y \in P; u,v \in B-B: x-u=y-v\right\}\right|.$$ 
On the other hand, $(a,b,c,d) \sim  (a',b',c',d')$ if and only if there exists $t \in B-B \cap C-C$ such that $t=a-a'=b-b'=c-c'=d-d'$, in which $a-a',b-b',d-d' \in B-B$ and $c-c' \in C-C$. Hence,
\[ \sum_{[a,b,c,d]} r^2([a,b,c,d])\leq \sum_{x\in (B-B)\cap (C-C)}r^3_{B-B}(x).r_{C-C}(x) \leq (E_4(B))^{3/4}(E_4(C))^{1/4}. \]
The last inequality is obtained by Holder's inequality. 

If follows from all above, we get 
\begin{align} \label{difineq1}
|B|^3|C| \ll N \ll (E_4(B))^{3/8}(E_4(C))^{1/8}\sqrt{\mathcal{X}}.
\end{align}
To bound the quantity $\mathcal{X}$, we use popularity of the difference  and dyadic localization. Namely, for some $\Delta \geq 1$ and some $T \subset (B-(B-C))$ one has

\begin{align*}
\mathcal{X} &\ll \dfrac{|B-C|^2}{|B|^2|C|^2}\left| \left\{b_1,b_2 \in B, c_1,c_2 \in C; u,v \in B-B: b_1-(c_1-u)=b_2-(c_2-v)\in B-C\right\} \right| \\
&\lesssim \dfrac{|B-C|^2}{|B|^2|C|^2}\Delta^2 \left|\left\{b_1,b_2 \in B, d_1,d_2 \in T \subset (B-(B-C)): b_1-d_1=b_2-d_2\in B-C \right\}\right| \\ 
&\leq \dfrac{|B-C|^2}{|B|^2|C|^2}\Delta^2\sum_{ w \in B-C} r_{B-T}(w)^2 \\
&\leq \dfrac{|B-C|^2}{|B|^2|C|^2}\Delta^2|B-C|^{1/2}\left(\sum_{w}r_{B-T}(w)^4\right)^{1/2}\\
&=\dfrac{|B-C|^{5/2}}{|B|^2|C|^2}\Delta^2 \sqrt{E_4(B,T)},
\end{align*}
where the last inequality is an application of Cauchy-Schwartz inequality. Now, applying the Lemma \ref{lemmasum}, one has
\[ \mathcal{X} \lesssim m^2 \cdot \min\left\{ \dfrac{|B-C|^{5/2}}{|B|^2|C|^2}\Delta^2 \dfrac{|f_1(A,B)|^{3/2}|T|}{|A|^{1/2}}, \dfrac{|B-C|^{5/2}}{|B|^2|C|^2}\Delta^2\frac{|f_1(A,B)||T|^{3/2}}{|A|^{1/2}} \right\}.\]
Note that $|T|\Delta \ll |B||B-C|$ and $|T|\Delta^2\ll E^{+}(B,B-C).$ Therefore,
\[ \mathcal{X} \lesssim m^2 \frac{|B-C|^{5/2}}{|B|^2|C|^2} \cdot \min\left\{\frac{|f_1(A,B)|^{3/2}}{|A|^{1/2}}E^{+}(B,B-C), \dfrac{|f_1(A,B)|}{|A|^{1/2}}|B||B-C|(E^{+}(B,B-C))^{1/2}\right\}.\]

Using the Theorem \ref{theo4}, we can obtain the upper bound for $E^+(B,B-C).$ More precisely, we have   
\begin{align*}
|A|^2E^{+}(B,B-C) &= |A|^2 \left|\left\{b_1,b_2,d_1,d_2 \in B^2 \times (B-C)^2: b_1-d_1=b_2-d_2 \right\}\right| \\
&\leq \big|\big\{ (a_1,a_2,f_1(a_1,b_1),f_1(a_2,b_2),d_1,d_2) \in A^2 \times f_1(A,B)^2 \times (B-C)^2: \\
& \quad \quad \quad \frac{f_1(a_1,b_1)}{g_1(a_1)}-h_1(a_1)-d_1=\frac{f_1(a_2,b_2)}{g_1(a_2)}-h_1(a_2)-d_2\big\}\big| \\
&\ll m^2|f_1(A,B)|^{3/2}|A|^{3/2}|B-C|^{3/2}.
\end{align*}
It gives us a upper bound of $E^{+}(B,B-C)$ as follows: 
\[ E^{+}(B,B-C) \ll m^2|f_1(A,B)|^{3/2}|A|^{-1/2}|B-C|^{3/2}.\]
Putting all inequality above together, we have
\begin{align} \label{difineq2}
\mathcal{X} \lesssim \frac{m^4|B-C|^{4}|f_1(A,B)|^{3}}{|B|^2|C|^2|A|}.
\end{align}
By \eqref{difineq1},\eqref{difineq2}, and the remark of Lemma \ref{lemmasum}, we get
\begin{align*} 
|B|^6|C|^2 &\ll (E_4(B))^{3/4}(E_4(C))^{1/4}\dfrac{m^4|B-C|^{4}|f_1(A,B)|^{3}}{|B|^2|C|^2|A|} \\
&\lesssim m^4 \cdot \frac{|f_1(A,B)|^{3/2}|f_2(D,C)|^{1/2}|B|^{9/4}|C|^{3/4}}{|A|^{3/4}|D|^{1/4}} \cdot \frac{m^4|B-C|^{4}|f_1(A,B)|^{3}}{|B|^2|C|^2|A|}
\end{align*}
Hence
$$|B|^{23/4}|C|^{13/4}|A|^{7/4}|D|^{1/4} \lesssim 
m^8|B-C|^4|f_1(A,B)|^{9/2}|f_2(D,C)|^{1/2}.$$
It follows that 
$$ \max\left\{ \left|f_1(A,B)\right|, \left|f_2(D,C)\right|, |B-C|\right\}  \gtrsim 
\frac{|B|^{23/36}|C|^{13/36}|A|^{7/36}|D|^{1/36}}{m^{8/9}}.$$
We complete the proof of Theorem \ref{maintheorem2}. 
\section{Proof of Theorems \ref{maintheorem3}, \ref{maintheorem4}, and \ref{maintheorem5}}
\begin{proof}[\bf Proof of Theorem \ref{maintheorem3}]
Let $P$ be a set of popular sums, defined as follows
\[ P=P(C):= \left\{ x \in C+C: r_{C+C}(x) \geq \epsilon.\frac{|C|^2}{|C+C|} \right\}, \]
where $\epsilon = \log{(C)}^{-1}.$ It implies 
\[ \left|\left\{ (c,c') \in C \times C: c+c' \in P \right\}\right| \geq (1-\epsilon)|C|^2.\] 
Furthermore, let $C'=C'(C):= \left\{ c'\in C:|\{ c" \in C: c'+c" \in P(C)\}| \geq (1-\epsilon)|C| \right\},$ so $|C'| \geq (1-\epsilon)|C|.$

Let $P' \subset C'-C'$ be popular by energy $E_{4/3}(C')$. Namely $x \in P'$ if for some $\Delta' \geq 1, \Delta'\leq r_{C'-C'}(x) \leq 2\Delta'$, and then
$$E_{4/3}(C') \gtrsim |P'|.\Delta'^{4/3}.$$
On the other hand, we also have $E_{4/3}(C')\gg E_{4/3}(C)$ (Lemma 8, \cite{RuSh1}), that is be used in the end of the proof.

Now, for $(b,c) \in C \times C, (a,d) \in B \times B$, consider the following equation
\begin{align} \label{sumeq1}
-c+b=(a+b)-(a+c)=(d+b)-(d+c).
\end{align}
Similar to the proof of Theorem \ref{maintheorem2}, let us make the popularity assumption as to the variables $a,b,c,d$. By the definition of the sets $C'$ and $P'$, it follows that the number of solutions $\phi$ of the equation ($\ref{sumeq1}$), when the different $b-c \in P'$ and all the four sums $x:=a+b, y:=a+c, u:=d+c,v:=d+b \in P$ is bounded from below as 
$$\phi \geq (1-4\epsilon)|P'| \Delta' |B|^2.$$
The equation \eqref{sumeq1} is invariant to a simultaneous shift of $b,c$ by $t$ and $d,a$ simultaneously by $-t$. We say $(a,b,c,d)$ is equivalent to $(a',b',c',d')$ if 
\begin{align*}
(a,b,c,d)=(a',b',c',d')+(t,-t,t,-t) \Leftrightarrow t=a-a'=b'-b=c-c'=d'-d,
\end{align*}
for some $t \in B-B \cap C-C$. \\
Each equivalent class $[a,b,c,d]$ yields a different solutions of the system of equations
\[ x,y,u,v \in P, w\in P': x-y=v-u=w. \]
Therefore, similar to the proof of Theorem \ref{maintheorem2}, by the Cauchy-Schwarz inequality, we get 

\begin{align*}
|B|^2|P'|\Delta' &\lesssim \left(\sum_{x\in B-B \,\cap\, C-C} r^2_{B-B}(x)r^2_{C-C}(x)\right)^{1/2} \sqrt{\left|\left\{x,y,u,v \in P, w \in P': x-y=v-u=w\right\}\right|} \\
&\ll E_4(B)^{1/4} E_4(C)^{1/4} \dfrac{|B+C|^2}{|B|^2|C|^2} \sqrt{\left| \mathcal{Y}\right|}.
\end{align*}
where $\mathcal{Y} = \{ (b_1,b_2,b_3,b_4 \in B, c_1,c_2,c_3,c_4 \in C: b_1+c_1-b_2-c_2=b_3+c_3-b_4-c_4 \in P' \}.$

There exist a popular subset $T \in B+C-C$ where $\forall d \in T, r_{B+C-C}(d) \approx \Delta$, for some $\Delta \geq 1$, such that one gets
\[ |B|^2\Delta'|P'| \lesssim (E_4(B))^{1/4} (E_4(C))^{1/4} \frac{|B+C|^2}{|B|^2|C|^2}|P'|^{1/4}\Delta.E_4(B,T)^{1/4}.\]
Applying Lemma \ref{lemmasum}, we obtain 
\[ |B|^2\Delta'|P'|^{3/4} \lesssim  m (E_4(B))^{1/4}(E_4(C))^{1/4}\frac{|B+C|^2}{|B|^2|C|^2} \frac{|f_1(A,B)|^{3/4}}{|A|^{1/4}}\left(|T|\Delta^2\right)^{1/2}.\]
On the other hand, we have 
\begin{align*} 
|T|\Delta^2 &\leq \left|\left\{(b_1,b_1') \in B, (c_1,c_2,c_1',c_2')\in C: b_1+c_1-c_2=b_1'+c_1'-c_2'\right\}\right|\\
&\ll \Delta_1^2 \left|\left\{(b_1,b_1')\in B, x_1,x_1' \in T_1: b_1+x_1=b_1'+x_1'\right\}\right|,
\end{align*}
where $T_1 \subset C-C$, with $r_{C-C}(x) \approx \Delta_1, \Delta \geq 1$. Again, applying Lemma \ref{lemmasum}, one gets
\[ \left(|T| \Delta^2 \right)^{1/2} \lesssim m \frac{|f_1(A,B)|^{3/4}|T_1|^{3/4}}{|A|^{1/4}} \Delta_1 \leq m \frac{|f_1(A,B)|^{3/4}}{|A|^{1/4}}(E_{4/3}(C))^{3/2}.\]
Collecting all inequalities above, we get 
\[ |B|^2(E_{4/3}(C'))^{3/4} \ll |B|^2 \Delta' |P'|^{3/4} \lesssim m^2(E_4(B))^{1/4} (E_4(C))^{1/4} \frac{|B+C|^2}{|B|^2|C|^2}\frac{|f_1(A,B)|^{3/2}}{|A|^{1/2}}\left(E_{4/3}(C)\right)^{3/4}.\]
Since $E_{4/3}(C) \ll E_{4/3}(C')$, we can cancel $(E_{4/3}(C))^{3/4}$ , and then we conclude that
\[ |B|^2 \lesssim m^2 (E_4(B))^{1/4} (E_4(C))^{1/4} \frac{|B+C|^2}{|B|^2|C|^2}\frac{|f_1(A,B)|^{3/2}}{|A|^{1/2}}.\]
Recall that
$$E_4(B) \lesssim m^4 \cdot \frac{|f_1(A,B)|^2|B|^3}{|A|},$$
$$E_4(C) \lesssim m^4 \cdot \frac{|f_2(D,C)|^2|C|^3}{|D|}.$$
Therefore,
\[ |B|^2 \lesssim 
m^4.\frac{|f_1(A,B)|^{1/2}|B|^{3/4}|f_2(D,C)|^{1/2}|C|^{3/4}}{|A|^{1/4}|D|^{1/4}}\frac{|B+C|^2}{|B|^2|C|^2}\frac{|f_1(A,B)|^{3/2}}{|A|^{1/2}}.
\] 
Hence,
\[ |B|^{13/4}|C|^{5/4}|A|^{3/4}|D|^{1/4} \lesssim m^4|B+C|^2|f_1(A,B)|^2|f_2(D,C)|^{1/2}.\]
It follows that
$$ \max\left\{|f_1(A,B)|,|f_2(D,C)|,|B+C|\right\} \gtrsim \frac{|B|^{13/18}|C|^{5/18}|A|^{1/6}|D|^{1/18}}{m^{8/9}}.$$
We complete the proof for Theorem \ref{maintheorem3}.  
\end{proof}

\begin{proof}[\bf Proof of theorem \ref{maintheorem4}]
Following the proof of Theorem \ref{maintheorem3}, we replace the sum (difference) operation on $B,C$ by the product (res. quotient) operation on $B,C$. Then applying the Lemma \ref{lemmapro} instead of Lemma \ref{lemmasum}, we obtain the Theorem \ref{maintheorem4}.
\end{proof}

\begin{proof}[\bf Proof of theorem \ref{maintheorem5}] 
By Corollaries \ref{cor1} and \ref{cor4} when $A \ll p^{3/5}$, we get
\[  |A|^{11/2} \lesssim |f(A,A)|^{5/2}|A+A|^{2},\]
\[  |A|^{11/2} \lesssim |f(A,A)|^{5/2}|A \cdot A|^{2}.\]
These inequalities imply
\[ |A|^{11} \lesssim |f(A,A)|^{5} \left( \min\left\{|A+A|,|A \cdot A|\right\}\right)^4. \]
Therefore, if $\min\left\{|A+A|,|A \cdot A| \right\} \leq |A|^{9/8-\epsilon},$ one gets $|A|^{11-9/2+4\epsilon} \lesssim |f(A,A)|^5.$ It follows that
$$|A|^{\frac{13}{10}+\frac{4\epsilon}{5}}\lesssim |f(A,A)|.$$
We complete the proof of theorem \ref{maintheorem5}.
\end{proof}

\end{document}